\documentclass[a4paper,12pt]{scrartcl}

\usepackage[english]{babel}

\usepackage{mathtools}
\mathtoolsset{showonlyrefs}

\usepackage{amssymb, amsthm, amsmath}

\usepackage{subcaption}
\usepackage{graphicx}

\usepackage{url}
\usepackage{hyperref}

\usepackage{enumerate}

\newcommand{\conn}{\psi}

\newcommand{\E}{\mathbb{E}}

\renewcommand{\P}{\mathbb{P}}
\renewcommand{\d}{\mathrm{d}}

\newcommand{\nats}{\mathbb{N}}

\newcommand{\ints}{\mathbb{Z}}
\newcommand{\reals}{\mathbb{R}}
\newcommand{\Pois}{\operatorname{Pois}}

\newcommand{\eps}{\varepsilon}

\newcommand{\del}{\partial}

\newcommand{\set}[2]{\big\{ #1 \mid #2 \big\}}

\newcommand{\inxi}[1]{\text{ in }\xi_{#1}}

\newtheorem{theorem}{Theorem}[section]

\newtheorem{lemma}[theorem]{Lemma}

\theoremstyle{definition}
\newtheorem{defi}[theorem]{Definition}

\theoremstyle{remark}

\newtheorem{remark}[theorem]{Remark}

\author{
  Niclas K\"upper\footnote{Supported by an EPSRC grant EP/W523914/1.}
  \and
  Mathew D. Penrose\footnote{Department of Mathematical Sciences, University of Bath, Bath BA2 7AY, UK. email: \texttt{m.d.penrose@bath.ac.uk}}
\footnote{Supported by EPSRC grant EP/T028653/1}}
\date{University of Bath \\
\today}
\title{Largest component and sharpness in  continuum percolation}

\begin{document}
\maketitle

\begin{abstract}
  We investigate the behavior of large connected components in the
  Poisson Random Connection model in non-critical regimes with any bounded connection function. We show that the asymptotic size of the largest component restricted to a window grows logarithmically in the volume of that window in the subcritical case, and linearly in the supercritical case. We also prove a sharpness result saying
  that the
  order of the cluster at the origin has an exponentially decaying tail in the subcritical regime.
\end{abstract}

\tableofcontents


\section{Introduction} \label{sec:Introduction}
Random geometric graphs are common stochastic models studied in statistical physics, wireless networking and as interesting mathematical objects in their own right. Most random geometric graphs of interest are defined on lattices or in the continuum. In this paper, we study the random connection model (RCM), a continuum random geometric graph that is built on a Poisson point process. It is characterized by a connection function $\conn$ that describes the probability of connecting two points.

To construct the RCM we first sample a Poisson point process in Euclidean space and then connect two points, independent of the rest of the process, with probability given by $\conn$. It is well known that this model (and many other random graphs) can have two distinct phases. By tuning the density of points $\lambda$, we can either find ourselves in the \emph{subcritical} phase with no long paths, or in the \emph{supercritical} phase with many long paths. There also exists a special \emph{critical} point $\lambda_c$  on the boundary of the two, whose properties are less well understood.

The restriction of the  $d$-dimensional RCM to a large box
$\Lambda_t := [-t/2, t/2]^d$ is also known as the
{\em soft random geometric graph} (SRGG). It can be viewed as a synthesis
of two classic finite random graph models, namely
the Erd\"os-Renyi random graph and the random geometric (Gilbert)
graph considered in the book \cite{penrose2003random}.


The ultimate goal of this paper is to determine the 
order of the largest component of the SRGG in the box $\Lambda_t$
(denoted $L(t)$), asymptotically as $t \to \infty$, for all non-critical values of $\lambda$. We shall demonstrate
that (i) if $\lambda < \lambda_c$ then $L(t)/ \log t$ converges in probability
to a certain positive constant, and (ii)
if $\lambda > \lambda_c$ then $L(t)/  t^d$ converges in $\mathcal L^1$
to another  positive constant,  namely the percolation probability.
(Naturally the limiting constants depend on $d$, $\psi(\cdot)$ and $\lambda$.)
Our results hold in all dimensions
$d \in \mathbb N$, but require $\psi$ to have finite range (i.e., bounded support).

For the Gilbert graph, i.e. the special case
where $\conn(\cdot) = \mathbf{1}\{\|\cdot\| \leq 1\}$
(with $\|\cdot\|$ denoting the Euclidean norm throughout this paper),
both (i) and (ii) (with convergence in probability) were already known by 2003
(see \cite{penrose2003random}), but the extension to more general connection
functions seems far from trivial and has required several more recent ideas.

We present our result on the supercritical case (Item (ii) above) in
Theorem \ref{th:super}.
Most of the argument for this case is covered elsewhere in \cite{pen_giant_comp_srgg,kupper2025,penrose2025supercritical} but those works  leave a gap unfilled, namely getting from convergence of expectations to convergence in $\mathcal L^1$ for $d \geq 3$;
here we fill that gap.

Our main concern is with Item (i) above, i.e., the subcritical case. The
main obstacle to adapting the argument for the Gilbert graph to
more general $\psi(\cdot)$ is to prove
what is often referred to as {\em sharpness} of the phase transition.
Roughly, sharpness states that in the subcritical phase large components become exponentially unlikely.

Our second main result, then, is a sharpness result for the RCM.
This is needed for our result on the largest component in the subcritical phase,
and is of interest in its own right.

For the Gilbert graph, sharpness was proved in \cite{penrose2003random}
by discretizing space and using Menshikov's sharpness result
\cite{menshikov1986coincidence}
for lattice percolation.
For the general  RCM, it seems much harder to get  a discretization method to work, because if $\mathbb R^d$ is divided into small boxes, for more general choices of $\psi$ one has to keep track of how many Poisson points lie in each box, rather than just whether the box contains a Poisson point or not.

More recently, other approaches to proving sharpness for
lattice percolation have become avaiable, and
our proof of sharpness is inspired by Duminil-Copin and Tassion's
proof of sharpness for lattice percolation in
\cite{duminilcopin2015new}.
Adapting this method to the  continuum RCM (without resorting to any discretization) seems to be a non-trivial matter, for reasons we shall
discuss further later on.

Yet another method of proving sharpness for lattice percolation was recently  found by Vanneuville \cite{vanneuville}.
Subsequent to the release of the first version \cite{v1} of the present paper,
Higgs \cite{higgs}  has given another proof of sharpness for the RCM
based on adapting Vanneuville's argument to the continuum;
also K\"upper has extended
the results and methods of the present paper
to {\em marked} random connection models in \cite{kupper2025}.

Our methods are heavily dependent on the connection function $\psi$ having
bounded support, and we would not expect that they
could easily be adapted to the case of unbounded support.
The proof of sharpness in \cite{higgs}, however,
{\em does} carry over to the case with unbounded support.


We now review some of the literature on the RCM and SRGG.
Early  results on the RCM can be found in Meester and Roy's book
on continuum percolation \cite{meester1996continuum}. These include the uniqueness of the infinite component (when it exists) and the agreement of two separate definitions of criticality.

Later,
Last and Ziesche showed a connection to the Ornstein-Zernike equation \cite{LastZiescheStationary}.  Heydenreich {\em et al.}
\cite{lace_exp_mean}
developed a lace expansion for the two-point function.
This yields, for
sufficiently high dimensions, the `triangle condition'
and hence continuity of the percolation function, along with information
about
certain critical exponents (see also
\cite{caicedo2023critical}).
In \cite{dickson2023expansion},
an expansion of the critical value $\lambda_c$ in the limit as $d\to\infty$
is developed.

Large deviations results for the largest component $L(t)$
in the supercritical Gilbert graph (for general $d$) were proved
by Penrose and Pisztora in \cite{penrose1996large}.
In the case where $\conn(x) = p\mathbf{1}\{\|x\|\leq 1\}$
with $p \in [0,1]$ it was shown in \cite{mit_giant_comp_srgg} for $d=2$ that the convergence of $L(t)/t^2$ in
the supercritical case holds almost surely,
and also that the critical parameters $\lambda_c(p)$ and $p_c(\lambda)$ are inverse functions  of each other.

Connection functions with infinite range have also been considered,
and results in that setting often have a different flavour.
Continuity of the percolation function has been established
in all dimensions for a class of infinite-range  connection functions,
allowing for marked points,
by M\"{o}nch \cite{monch}.
Large deviations results for the largest component
for a  certain class of supercritical
`kernel-based' connection functions with infinite range allowing marked points
are given by Jorritsma {\em et al.} in \cite{jorritsma2025large}.

The rest of this paper is organized as follows.
In Section \ref{sec:MainRes} we construct the model, state the main theorem
(Theorem \ref{thm:main_large}) on the largest component
in the subcritical phase,
and describe the existing tools we require.
In Section \ref{sec:sharp} we prove
(and further discuss) sharpness for the RCM (Theorem \ref{thm:sharpness}).
%
The limiting constant in Theorem \ref{thm:main_large} involves an
exponent we call the
{\em volume decay rate} (see Definition \ref{def:cor-length}).
In Section \ref{sec:LargeComp},
we present Theorem \ref{thm:zeta-prop}, on existence and basic
properties of the volume decay rate, and then
complete the proof of
Theorem \ref{thm:main_large}.
Finally in Section \ref{s:supercritical}
we prove the law of large numbers
for the {\em supercritical} largest component of the SRGG
(Theorem \ref{th:super}).

\begin{figure}[h]
  \centering
  \begin{subfigure}{.49\textwidth}
    \centering
    \includegraphics[width=.7\linewidth]{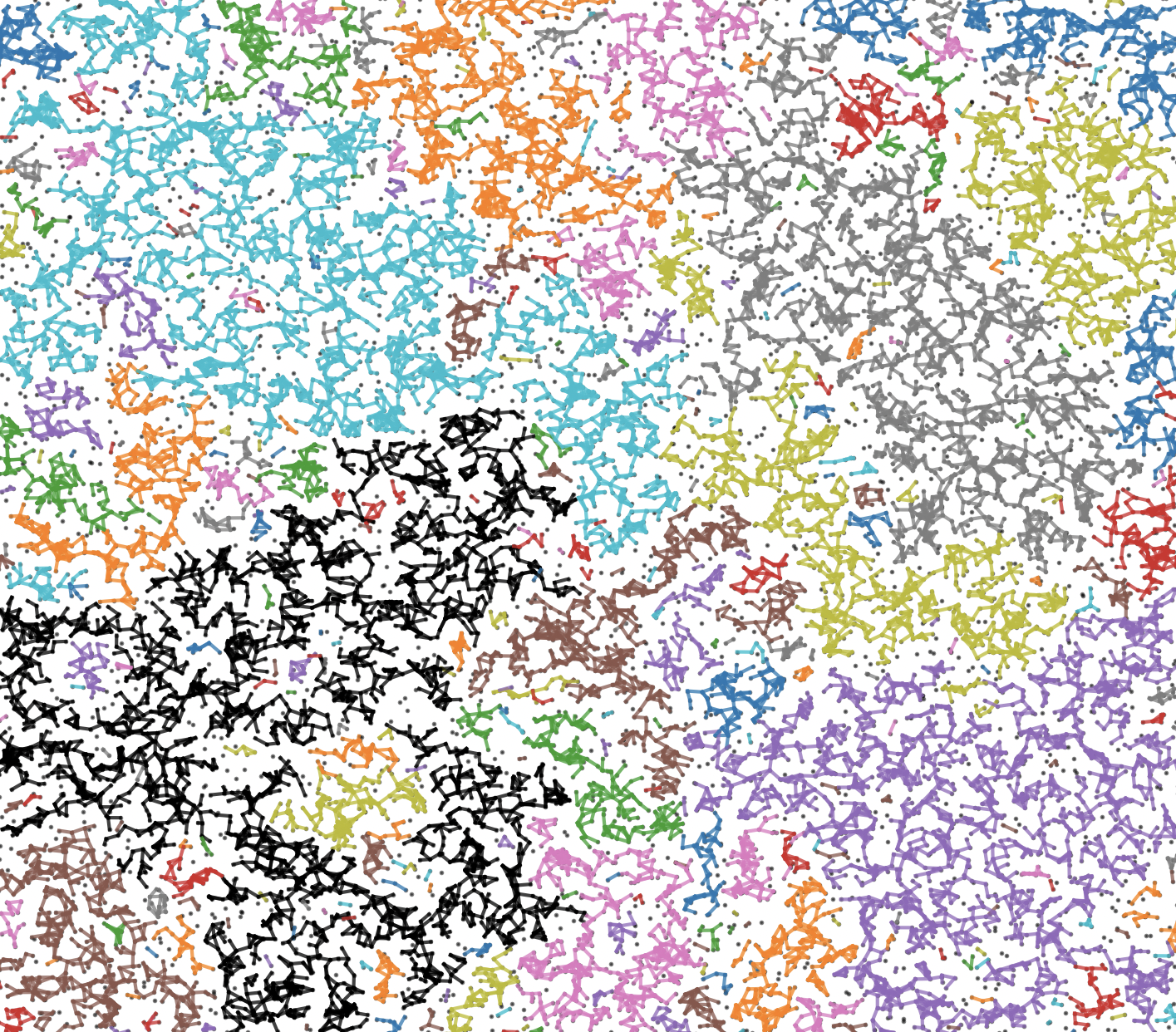}
  \end{subfigure}
  \begin{subfigure}{.49\textwidth}
    \centering
    \includegraphics[width=.7\linewidth]{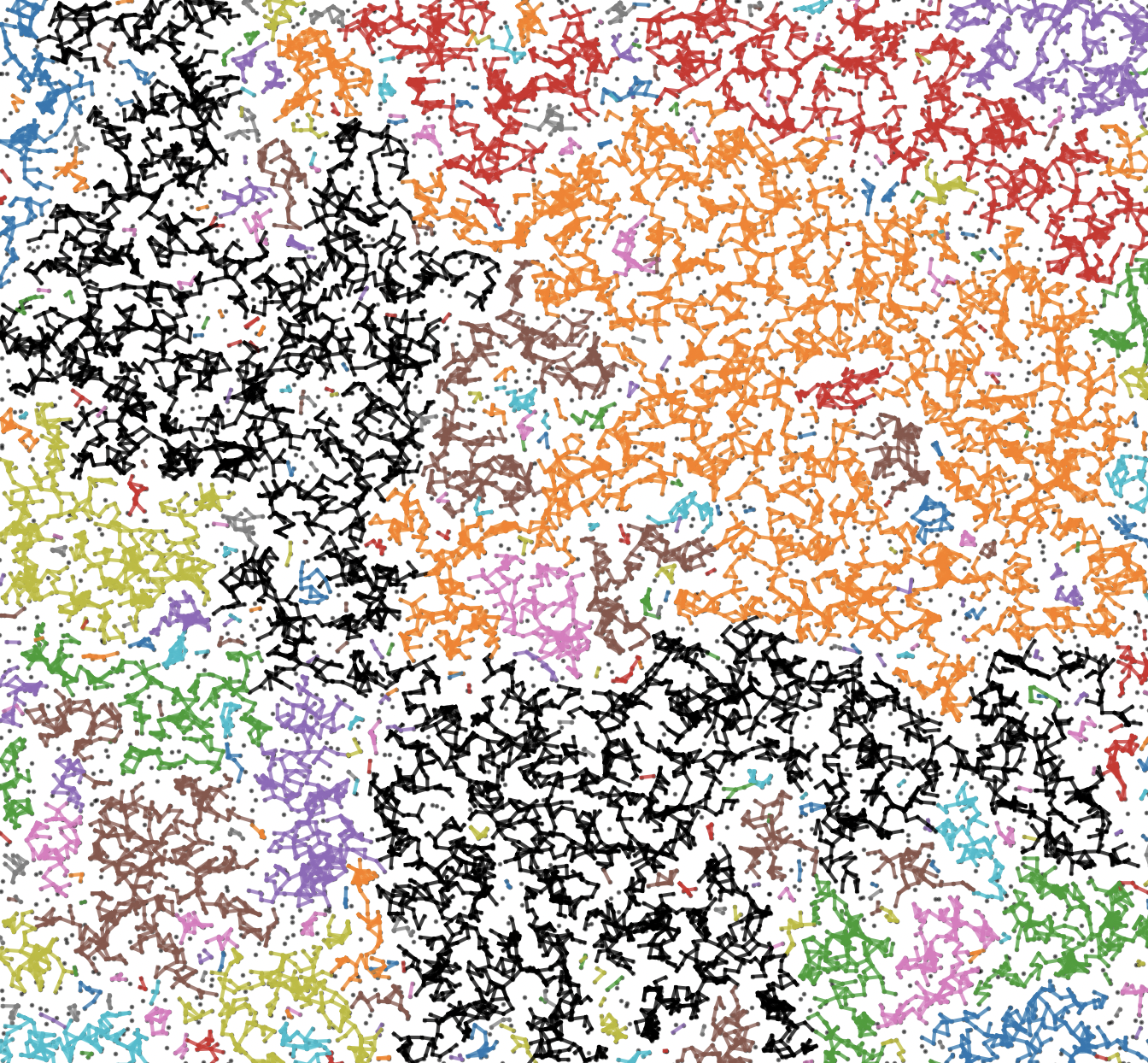}
  \end{subfigure}
  \caption{Two instances of a near-critical SRGG with 
    $\lambda=2.5, \psi (x) =0.35 \mathbf{1}\{\|x\| \leq 1\}$
  and an observation window of side-length $80$. The largest connected components are colored black.}
\end{figure}

\section{Background and Main Result} \label{sec:MainRes}
\subsection{Definitions and Overview}
\label{ss:DefOver}

To construct the Random Connection Model, we will use a formalism akin to \cite{LastZiescheStationary} and \cite{lace_exp_mean}. We introduce a measurable connection function $\psi: \mathbb{R}^d \rightarrow [0,1]$ that is symmetric, i.e., $\psi(x) = \psi(-x)$, and has bounded support,
i.e., $\text{supp}(\psi) \subset \Lambda_K$ for some $K \in (0,\infty)$.

We want to rigorously establish a method to take any point process and turn it into a Random Connection Model. To that end let $\eta$ be any well-behaved point process in $\reals^d$. We assume there exists an enumeration $(X_i)_{i \in I}$ of $\eta$, for some index set $I \subset \mathbb N$.
Let $<$ denote any strict total order on $\reals^d$, such as the lexicographic ordering. Let $(U_{i,j})_{i, j\in \nats}$ be independent uniform random variables on $[0, 1]$ independent of $\eta$. Then we will define $\xi[\eta]$ by
\begin{equation}
  \xi[\eta] := \big\{
    (\{X_i, X_j\}, U_{i,j})
    \mid
    X_i, X_j\in \eta,
    X_i<X_j
  \big\}.
\end{equation}
We call $\xi[\eta]$ an \emph{independent edge marking of $\eta$}. We note that the independent edge marking $\xi[\eta]$, together with a choice of $\conn$, contains sufficient information to extract the geometric random graph of interest where
\begin{equation}
  \begin{aligned}
    (V, E)
    =
    (
      \eta,
      \set
      {\{X_i, X_j\}}
      {X_i<X_j, U_{i,j}\leq \conn(X_i-X_j), \forall i,j\in I 
    })
    .
  \end{aligned}
\end{equation}
See \cite{lace_exp_mean} for more information.

Let $\mathcal P_\lambda$ be a homogeneous Poisson point process on $\reals^d$ with intensity $\lambda>0$.
We will write $\xi_\lambda := \xi[\mathcal P_\lambda]$ to simplify notation.
We will require to couple the Poisson RCM across varying $\lambda$ in
such a way that $\xi_{\lambda_1}$ is a subgraph of $\xi_{\lambda_2}$
whenever $\lambda_1 < \lambda_2$.
This can be achieved by choosing some maximum $\lambda_M$ and taking the coupling across $\mathcal P_\lambda, 0 \leq \lambda \leq \lambda_M$ as follows.
Assume there is a further collection of random variables
$W_i, i \in \mathbb N$ that are uniformly distributed over $[0,\lambda_M]$,
independent of each other and of everything else. Enumerate
$\mathcal P_{\lambda_M} = (X_i)_{i \in \mathbb N}$ (for example, in order of increasing distance from the origin). 
Give each point
$X_i$ of $\mathcal P_{\lambda_M}$ the mark $W_i$, and for $\lambda \in [0, \lambda_M]$ set $\mathcal P_{\lambda} = \{X_{i}: W_i \leq \lambda \}$.
Thus if $\lambda  < \lambda_M$ we have $\mathcal P_{\lambda}  =
\{X_i: i \in I\}$ for a certain random proper subset $I$ of $\mathbb N$.


Next, we require a method to add additional points to the Poisson point process without disturbing the already existing edges. This can be achieved by sampling additional random variables $(U_{-1, n})_{n\in\nats}$, and assigning these to the edges between the added point and all other points, in effect viewing the additional point as before all others in the total order. If for some $x\in\reals^d$ we wish to add $x$ to $\xi_\lambda$ we write $\xi_\lambda^x := \xi[\mathcal P_\lambda\cup \{x\}]$, by some abuse of notation.
For further additional points we simply iterate the first index of $U_{i, j}$ down by one.

Finally we define connected components in the usual way for the graph induced by the independent edge marking $\xi_\lambda$. We write $C_x := C_x(\lambda) := C(x, \xi^x_\lambda)$ for the connected component of $x\in\reals^d$.

The object of interest is the \emph{largest connected component} of this graph when restricted to a large box (see Definition \ref{def:large-comp}). To this end, let $\Lambda_l$ denote the box $[-l/2, l/2]^d$ with side-length $l$. We also define the shifted set $\Lambda_l(x) := \Lambda_l + x$ for all $x\in\reals^d$. The largest connected component within this box is represented as $L_1(\xi_{\lambda} \cap \Lambda_l)$. The number of nodes in this component is denoted by $|L_1(\xi_{\lambda} \cap \Lambda_l)|$. More generally, we use $|\cdot|$ to indicate the number of vertices in a graph.

To study the largest component we will require the tools of percolation theory. For the purposes of this theory, it will be convenient to introduce an extra point into our Poisson process. Let $o$ denote the origin of $\mathbb{R}^d$. We define $\xi_{\lambda}^o$ as the random geometric graph described earlier, but with the inclusion of an additional point at the origin. The edges between $o$ and all other points are added in the same manner as before, independently of everything else. We begin our discussion by defining a connection between points and sets.

\begin{defi}[Connection]
  Let $x, y\in\mathcal{P}_\lambda$. We say $x$ and $y$ are \emph{connected} if a finite sequence of edges $(e_i)_{i\leq n}$ exists in $\xi_\lambda$ such that $e_i = \{z_i, z_{i+1}\}$ where $z_0=x$ and $z_{n+1} = y$. We denote this event as $\{x \leftrightarrow y \text{ in } \xi_{\lambda}\}$.
  We call two sets $A, B \subset \mathbb{R}^d$ \emph{connected} if there exist points $x \in A \cap \mathcal{P}_\lambda$ and $y \in B \cap \mathcal{P}_\lambda$ such that $x \leftrightarrow y$. We denote this event as $A \leftrightarrow B$. Lastly, for the event $\{x\} \leftrightarrow A$ in $\xi_\lambda^x$, we simply write $x \leftrightarrow A$ in $\xi_\lambda^x$.
\end{defi}

\begin{defi}[Percolation Probability]
  Given $\lambda \in (0, \infty)$ and given a connection function $\psi$, the $t$-percolation probability is defined as
  \begin{equation}
    \theta_t(\lambda) = \theta_t(\lambda, \psi) = \mathbb{P}[o \leftrightarrow (\mathbb R^d \setminus \Lambda_t) \text{ in } \xi_{\lambda}^o].
  \end{equation}
  We define the percolation event as $\{o\leftrightarrow \infty\} := \bigcap_{t>0}\{o\leftrightarrow (\mathbb R^d \setminus \Lambda_t)\}$.
  We define the percolation probability as
  \begin{equation}
    \theta(\lambda) = \theta(\lambda, \psi) = \mathbb{P}_{\lambda}[o \leftrightarrow \infty \text{ in } \xi_{\lambda}^o] =
    \lim_{t \to \infty} \theta_t(\lambda).
  \end{equation}
\end{defi}

\begin{remark}
  The quantity $\theta(\lambda)$ is well defined. The $t$-percolation probability is monotonically decreasing in $t$ and bounded from below by $0$, ensuring the limit exists. The quantity $\theta(\lambda)$ can also be interpreted as the probability that the origin is connected to the unique infinite component, should it exist. Uniqueness of the infinite component is established in \cite{meester1996continuum}.
\end{remark}

\begin{defi}[Criticality]
  Given a connection function $\psi$, the critical parameter $\lambda_c$ is defined as
  \begin{align}
    \lambda_c = \lambda_c(\psi) = \inf\{\lambda \mid \theta(\lambda) > 0\}.
    \label{e:crit}
  \end{align}
  A value of $\lambda$ is said to be \emph{subcritical} if $\lambda < \lambda_c$ and \emph{supercritical} if $\lambda > \lambda_c$.
  We define $\inf \varnothing = \infty$, so if $\psi \equiv 0$, then $\lambda_c = \infty$. Also, if $d=1$ and $\psi $ has bounded support then
  it is well known that $\lambda_c = \infty$.
\end{defi}

\begin{defi}[Largest Component] \label{def:large-comp}
  The largest component, as measured by the number of vertices, in a bounded Borel set $B\subset \mathbb{R}^d$ is denoted as $L_1(\xi_{\lambda} \cap B)$. In cases of multiple largest components, the one containing the largest point with respect to lexicographic ordering $<$ is chosen. The order of this component is denoted by $|L_1(\xi_{\lambda} \cap B)|$. When $B$ and $\xi$ are clear from the context, we will simply write $L_1$.
\end{defi}

\begin{defi}[Volume decay rate] \label{def:cor-length}
  The {\em volume decay rate}
  for parameters $(\lambda, \conn)$ is defined as
  \begin{align}
    \zeta(\lambda, \conn)
    & :=
    \lim_{n\to\infty}
    -\frac1n \log \P[|C_o|= n\text{ in }\xi_{\lambda}^o] \label{eq:def-cor-1}
    \\
    & =
    \lim_{n\to\infty}
    -\frac1n \log \P[
      n\leq|C_o|< \infty \text{ in }\xi_{\lambda}^o \label{eq:def-cor-2}
    ].
  \end{align}
\end{defi}
In Theorem \ref{thm:zeta-prop} we shall show
that the limits \eqref{eq:def-cor-1} and \eqref{eq:def-cor-2} are well defined, finite and equal when $\lambda<\lambda_c$, and give some properties of $\zeta(\cdot,\psi)$.
Note that $\log$ will always refer to the natural logarithm.

Here `volume' is shorthand for the number of vertices in a component.
One could seek to define a {\em diameter decay rate}
or {\em inverse correlation length} $\tilde{\zeta}(\lambda, \psi)$
by a similar formula to \eqref{eq:def-cor-2} but with the event
$\{n \leq |C_0| < \infty\}$ replaced by
$\{o \leftrightarrow (\mathbb R^d \setminus \Lambda_{2n})\}$, assuming
the limit exists (existence of such a limit is known for lattice models).

\subsection{Main Theorem}
\begin{theorem}[Main Theorem: Large Components]\label{thm:main_large}
  Consider the RCM with connection function $\conn$  having bounded support and subcritical intensity $\lambda\in(0, \lambda_c)$. Then
  $$
  \frac{|L_1(\xi_{\lambda}\cap \Lambda_s)|}{\log s}
  \to
  \frac{d}{\zeta(\lambda, \psi)}
  \textup{ in probability},
  $$
  as $s\rightarrow\infty$.
\end{theorem}
\begin{remark}
  A weaker version of this result is more straightforward to prove,
  namely
  \begin{equation}
    \frac{|L_1(\xi_{\lambda}\cap\Lambda_s)|}{s^d}
    \to 0
    \quad \text{in probability whenever }  \theta(\lambda) =0.
  \end{equation}
  This was proved for $d=2$ in \cite[Theorem 1.1]{pen_giant_comp_srgg},
  and the proof readily carries over to general dimensions.

\end{remark}

To prove the main result, we will first need to prove that the $t$-percolation probability decays exponentially fast in $t$. This is done in Section \ref{sec:sharp}, where we will prove a stronger statement commonly called \emph{sharpness} of the phase transition. Sharpness refers to an exponentially decaying
upper bound on the probability from the origin $o$ to a distance $t$ and a lower bound on the supercritical percolation probability.

\subsection{Standard Tools}\label{ssec:standard_tools}

Next we discuss the standard tools we can use with this formalism. These tools were adapted to this context in \cite{LastZiescheStationary} and \cite{lace_exp_mean}.

\paragraph{The Mecke Equation.}\label{par:mecke}
The first crucial tool
is the Mecke equation
\cite{lace_exp_mean}, which allows us to calculate the expected values of sums. 
Given an $m\in\nats$, a measurable function $f$ of the correct domain mapping to $\reals_{\geq0}$
and an independent edge marking $\xi_\lambda$, the Mecke equation for $\xi_\lambda$ states
$$
\E\bigg[
  \sum_{\vec x\in\mathcal P_{\lambda}^{(m)}}
f(\xi_\lambda, \vec x)\bigg]
=
\lambda^m \int_{(\reals^d)^m}
\E\bigg[
f(\xi_{\lambda}^{x_1,\dots, x_m}, \vec x)\bigg]
\d\vec x,
$$
where $\mathcal P_\lambda^{(m)}=\{\vec x\ |\ x_i\in\mathcal P_\lambda, :x_i\neq x_j,\forall i\neq j\}$ and $\vec x=(x_1,\dots,x_m)$,
and $\xi_\lambda^{x_1,\ldots,x_m}
:= \xi [ \mathcal P_\lambda \cup \{x_1,\ldots,x_m\}]$.
In the case where $m=1$ we obtain the simpler form
$$
\E\bigg[
  \sum_{ x\in\mathcal P_{\lambda}}
f(\xi_{\lambda}, x)\bigg]
=
\lambda\int_{\reals^d}
\E\bigg[
f(\xi_{\lambda}^{x}, x)\bigg]
\d x,
$$
Note that if we force additional points into our point process we have to modify the Mecke equation as follows:
\begin{align*}
  \E\bigg[
    \sum_{ x\in\mathcal P_{\lambda}^y}
  f(\xi_{\lambda}^y, x)\bigg]
  =
  \E\bigg[
    f(\xi_{\lambda}^{y}, y)
  \bigg]
  +
  \lambda \int_{\mathbb R^d} \E\bigg[
    f(\xi_{\lambda}^{x,y}, x)
  \bigg]
  \d x
  ,
\end{align*}
where $y\in\reals^d$.

\paragraph{Margulis-Russo formula.}\label{par:Margulis}
This formula is standard on lattice models, and has been expanded to the RCM in \cite{LastZiescheStationary}. Let $f$ be a function with $\xi[\eta]$ as input. For any Borel set $\Lambda\subset\reals^d$ and an edge marking $\xi[\eta]$ of $\eta$ we define the restriction of $\xi[\eta]$ to the set $\Lambda$ by
$$
\xi[\eta\cap \Lambda] :=
\set{
  (\{ x, y\}, u) \in\xi
}{ \{x,y\}\subset \eta\cap \Lambda
}.
$$
We say a function $f$ lives on a Borel set $\Lambda\subset \reals^d$ if $f(\xi[\eta\cap\Lambda])=f(\xi[\eta])$ almost surely.  Now assume $f$ lives on some bounded $\Lambda$ and there exists some $\lambda_0>0$ such that $\E[f(\xi_{\lambda_0})]<\infty$. Then the Margulis-Russo formula states that for every $\lambda\leq \lambda_0$ we have
$$
\frac{\del}{\del\lambda} \E[f(\xi_{\lambda})]=
\int_{\Lambda}
\E[f(\xi^x_{\lambda}) - f(\xi_{\lambda})]\d x.
$$

\section{Exponential decay and sharpness} \label{sec:sharp}
The term {\em sharpness} of the phase transition originates from the coincidence of two different definitions of criticality. The first is $\lambda_c$ as
we have defined it in \eqref{e:crit}. The second $\lambda^E_c$ is the point where the quantity $\E_\lambda[|C_o|]$ first becomes infinite. It is immediate that $\lambda_c \geq \lambda_c^E$, since an infinite path from the origin implies an infinite component. If $\lambda_c = \lambda_c^E$, the phase transition might be considered to be sharp, but in fact we use
{\em sharpness} 
for the stronger statement that for all $\lambda < \lambda_c$ the size of
$C_0$ has an exponentially decaying tail (rather than just a finite
mean), in keeping with
common recent usage of the term.

%
As mentioned earlier, we shall prove sharpness by adapting
%
%
the method of  \cite{duminilcopin2015new} to the RCM.
The idea of that proof is to define a functional (traditionally denoted $\varphi$),  for each value of the density parameter,
whose domain is (in the lattice setting of
\cite{duminilcopin2015new}) the collection of all finite
subsets $S$ of the lattice containing the origin.
The functional $\varphi$ outputs
the expected number of sites just outside
$S$ that can be reached through $S$.
If the range of $\varphi$ includes
a value less than 1, then by a  branching process domination
argument one can show exponential decay, while if it does not,
then using the Margulis-Russo formula one can
show the percolation
probability is growing linearly in the density
parameter.

The main difficulty in adapting this to the RCM
setting
lies in finding an appropriate definition and domain of
a $\varphi$-function.
We invite the reader familiar with the method of
\cite{duminilcopin2015new} to guess what the definition and
domain of our function $\varphi$ will be before reading on;
this was not at all obvious to the authors
{\em a priori}.
\subsection{Notation}

Recall that we write $C_x=C(x, \xi_\lambda^x)$ for the cluster of $x$, that is, the collection of all $y\in\mathcal P_\lambda$ such that $x\leftrightarrow y$ in $\xi_\lambda^x$. To be clear, we might write that some event $A$ happens in $\xi[C(x,\xi_\lambda^x)]$, in which case we \emph{do not} resample the edges between vertices in $C(x,\xi_\lambda^x)$.

We call a measurable function $f:\reals^d\to[0,1]$ a \emph{thinning function} if it has bounded support. Let $\mathcal T$ denote the set of all such functions. Given a point process $\eta$ with (independent) markings $U_x\sim \text{Unif}([0,1])$ we define
\begin{equation}
  f_*\eta := \{x\in\eta\mid U_x\leq f(x)\}
\end{equation}
to be the $f$-thinning of $\eta$.
In the case where $\eta = \mathcal P_\lambda$, we take $U_{X_i}
= W_{i}/\lambda$ where the $W_i$ were introduced in Section \ref{ss:DefOver}.


\subsection{Statement of the sharpness result}
To prove sharpness, we define a functional that takes a thinning function as an input and returns a real number:
\begin{equation}\label{e:defphi}
  \varphi_\lambda(f) :=
  \lambda\int_{\reals^d}
  (1-f(y))
  \P[o\leftrightarrow y
    \text{ in }
  \xi[f_*\mathcal P_\lambda\cup\{o,y\}]]
  \d y.
\end{equation}
This value can be interpreted as the expected number of points in $\mathcal P_\lambda\setminus f_*\mathcal P_\lambda$ that can be reached from the origin only using points in $f_*\mathcal P_\lambda$.
We can now define a new critical parameter
\begin{equation}
  \label{e:tlacdef}
  \tilde\lambda_c:=
  \sup\{\lambda\mid\exists f\in\mathcal T:\varphi_\lambda(f)<1\}.
\end{equation}

\begin{figure}[h]
  \centering
  \includegraphics[width=.7\linewidth]{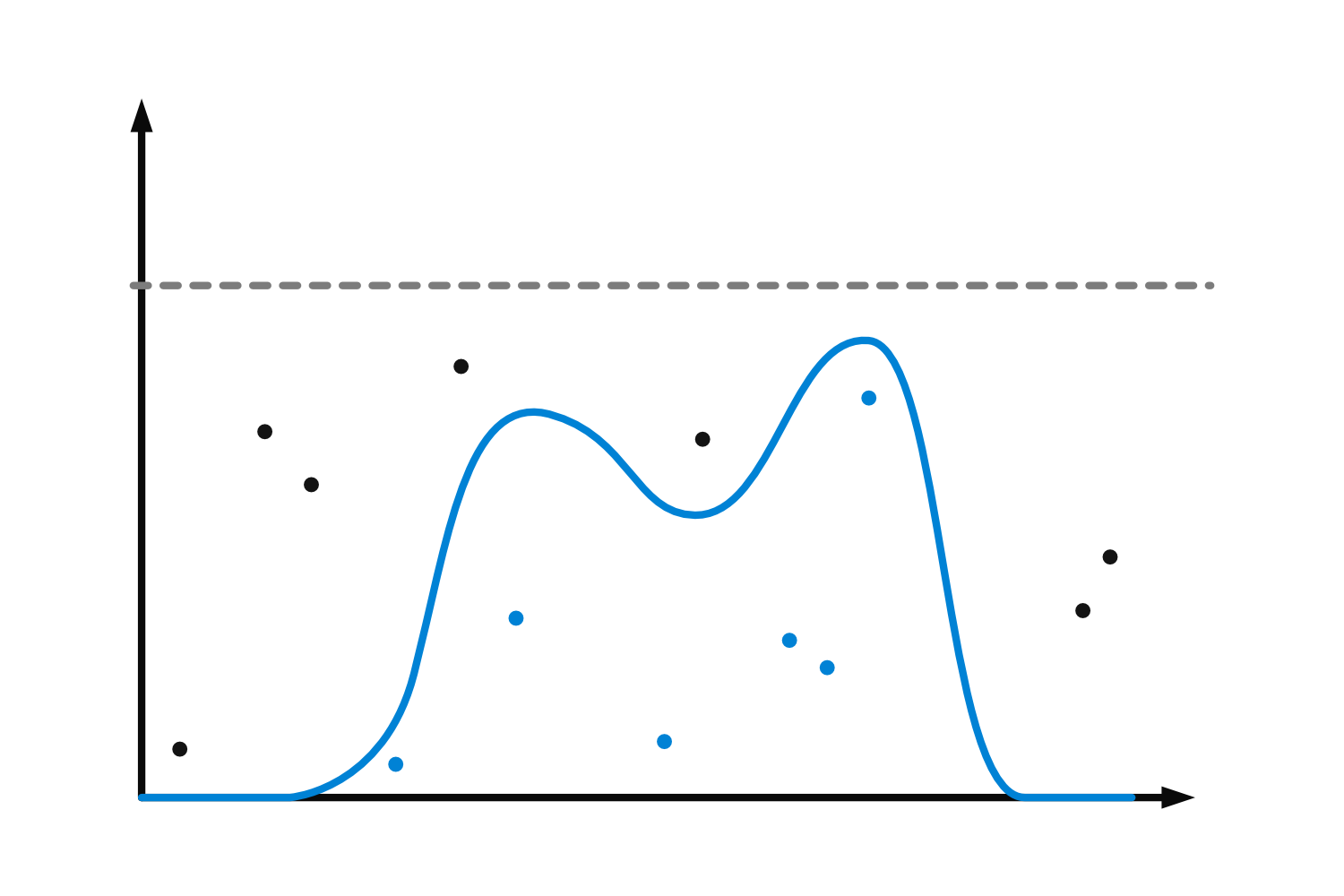}
  \caption{The $y$-axis describes the intensity of each point and the $x$-axis represents space. For some function $f$ and a Poisson point process $\eta$ of intensity $\lambda$. The points marked in blue (below the curve) are the $f$-thinning of $\eta$.}
\end{figure}

\begin{theorem}[Sharpness]\label{thm:sharpness}
  For any $d\geq1$ and any connection function $\psi$ with bounded support,
  it holds that $\tilde\lambda_c=\lambda_c$ and
  \begin{enumerate}[(I)]
    \item \label{itm:first} If $0 < \lambda<\lambda_c$ then there exists some $c= c(\lambda) >0$ such that for all $t >0$,
      $$
      \theta_t(\lambda,\psi)\leq e^{-ct}.
      $$
    \item \label{itm:second} For all $\lambda>\lambda_c$ we have
      $$\theta(\lambda, \psi)\geq \frac{\lambda-\lambda_c}{\lambda}.$$
  \end{enumerate}
\end{theorem}
%
\begin{remark}
  By using the definition of $\varphi_\lambda$,
  choosing $f\equiv 0$ we obtain that
  \begin{equation}
    \varphi_{\lambda}(f) = \lambda \int_{\mathbb R^d} \conn(y) \d y
    \quad\text{ and hence }\quad
    \lambda_c\geq \frac{\mathbf{1}}{\int_{\mathbb R^d} \conn(y)\d y}.
  \end{equation}
  This 
  lower bound on $\lambda_c$ is proved by other means in \cite{penrose1991} or \cite{meester1996continuum}. It shows in particular that $\lambda_c >0$, so
  the set of values of $\lambda$ for which Theorem \ref{thm:main_large} applies is non-empty.
\end{remark}
\begin{remark}
  Theorem \ref{thm:sharpness} demonstrates `exponential decay in diameter' of $C_o$ for all $\lambda < \lambda_c$. In Section \ref{ss:expvol} we
  shall
  improve this to `exponential decay in volume' (i.e., in
  the number of vertices of $C_0$), which implies
  $\lambda_c = \lambda_c^E$.
\end{remark}

\begin{remark}
  \label{rk:sharpcontext}
  In the case where $\psi(\cdot)$ is a decreasing function of $\|\cdot\|$ (not necessarily with bounded support),
  the result $\lambda_c = \lambda_c^E$ was proved in
  \cite{meester1996continuum}.
  More recently, $\lambda_c = \lambda_c^E$ was proved for
  a more general class of {\em marked} random connection
  models
  in \cite{caicedo2023critical}, along with a lower bound
  similar to our Theorem \ref{thm:sharpness},
  Item \ref{itm:second}. Our approach yields a simpler proof of $\lambda_c= \lambda_c^E$ than those works, along with
  the exponential decay results of Theorem \ref{thm:sharpness},
  Item \ref{itm:first}, and Lemma \ref{lem:vol_decay} below,
  albeit only for $\psi$ with bounded support.

  A related model of continuum percolation is the {\em Boolean model}
  where balls of fixed or random radius, or more general random shapes, are
  centred on the points of our Poisson point process. Sharpness results for this model have been proved for balls of possibly unbounded radius
  subject to a moment condition \cite{dembin2023almost,duminil2020subcritical}.
  Ziesche
  \cite{ziesche} proves sharpness
  for a more general random shapes subject to a boundedness condition, using the method of
  \cite{duminilcopin2015new};
  we note that the $\varphi$-function
  used by \cite{ziesche} has as its domain a class of subsets of $\mathbb R^d$, whereas in our case we seem to need the domain of $\varphi$ to be a class of thinning functions.
\end{remark}

\subsection{Proof of Theorem \ref{thm:sharpness}}
While our proof of Theorem \ref{thm:sharpness} follows the general scheme of
the proof of sharpness for lattice percolation in \cite{duminilcopin2015new},
it  requires some extra technical ingredients due to working in the continuum,
notably the notions of ghost-vertex and stopping sets below.

Recall that in Theorem \ref{thm:sharpness} we assume $\psi$ has bounded support, i.e. supp$(\psi) \subset \Lambda_K$ for some $K \in (0,\infty)$. We shall prove
Theorem \ref{thm:sharpness} only in the case $K=1$; the result for general $K$
is then easily obtained either by adapting the proof or by using scaling.

We show items \ref{itm:first} and \ref{itm:second} for $\tilde\lambda_c$.
It follows immediately from \ref{itm:first} and \ref{itm:second} that $\tilde\lambda_c=\lambda_c $.
We fix $\lambda\in (0,\lambda_c)$ and set $\eta = \mathcal P_\lambda$.

For any given $t$ we identify $\mathbb R^d \setminus \Lambda_{t+2}$
with an abstract `ghost vertex' $g_t$ which is added to $\xi$ in the same way as other added vertices. We add an edge between $x\in\eta\cap\Lambda_t$ and $g_t$ with probability $\conn(x,g_t)=\mathbf{1}\{x\in\Lambda_{t+2}\setminus \Lambda_t\}$. Then
\begin{equation}\label{eq:ghost}
  \theta_t(\lambda, \psi)=\P[o\leftrightarrow
  (\mathbb R^d \setminus \Lambda_t)]
  = \P[o\leftrightarrow g_t],
\end{equation}
since any vertex in $\Lambda_{t+2} \setminus \Lambda_t$ is by definition of $\conn(\cdot, g_t)$ connected to $g_t$.

We will require the following lemma from \cite[Corollary 3.4]{lace_exp_mean}.
\begin{lemma}[Stopping Set Lemma]\label{cor:stop}
  Let $\lambda >0$. Let $v,u,x \in\reals^d$ be distinct. Then, for $\P[C(v,\xi_\lambda^v) \in \cdot]$-a.e. $A$,
  \begin{equation}
    \P\big[
      u\leftrightarrow x \textup{ in } \xi[(\eta\setminus A)\cup \{u, x\}]
      \mid C(v,\xi^v) = A
    \big]
    =
    \P\big[
      u\leftrightarrow x\textup{ in } \xi[f^A_*\eta\cup\{u, x\}]
    \big],
  \end{equation}
  where $f^A(\cdot) = \prod_{z\in A} (1-\conn(z-\cdot))$ is the probability of not connecting to $A$.
\end{lemma}
Note that for general applications $f^A$ might not have bounded support.

\subsubsection{Item \ref{itm:first}}
Assume that $\lambda<\tilde\lambda_c$. Then, by definition, we can and do choose a thinning function $f$ such that $\varphi_\lambda(f)<1$. Choose $L>1$ such that $\text{supp}f\subset \Lambda_{L-2}$. This implies that all vertices of $f_*\eta$ are in $\Lambda_{L-2}$. We write $\mathcal C := C(o, \xi[f_*\eta\cup\{o\}])$ for the component of the origin of the thinned process.

We remark that for any finite $C \subset \Lambda_{L-2}$,
event
$\{o\leftrightarrow x \text{ in } \xi[C\cup \{x\}]\}\cap\{\mathcal C=C\}$ is exactly the same event as $\{C\sim x\}\cap\{\mathcal C=C\}$, in particular, we remind the reader that we do not resample the edges when we write $\xi[C]$.

Let $k\in\nats$. The event $\{o\leftrightarrow (\mathbb R^d \setminus \Lambda_{kL}) \textup{ in } \xi_\lambda^o\}$ holds if and only if there exists some
$x\in\eta\setminus f_*\eta$ such that $o\leftrightarrow x$ in
$\xi[f_*\eta\cup\{o,x\}]$ and $x\leftrightarrow g_{KL}$
(the ghost-vertex)
off $\mathcal C$.
Then, by applying Markov's inequality followed by the Mecke equation we find that
\begin{equation}
  \begin{aligned}
    \theta_{kL}(\lambda)
    & \leq
    \E\Big[\sum_{x\in\eta\setminus f_*\eta}
      \mathbf{1}\{o\leftrightarrow x \text{ in } \xi[f_*\eta\cup\{o,x\}]\}
      \mathbf{1}\{x\leftrightarrow
      (\mathbb R^d \setminus \Lambda_{kL}) \text{ off }\mathcal C\}
    \Big]                                        \\
    & =\lambda\int_{\reals^d}
    \P[
      o\leftrightarrow x \text{ in } \xi[f_*\eta\cup\{o,x\}],
      x\leftrightarrow g_{KL} 
      \text{ off } \mathcal C
    ]
    (1-f(x))\d x                                   \\
    & =
    \lambda\int_{\reals^d}\int
    \P[
      o\leftrightarrow x \text{ in } \xi[C\cup \{x\}],
      x\leftrightarrow g_{k L} 
      \text{ in } \xi[(\eta\setminus C)\cup \{x\}]
      \mid
      \mathcal C = C
    ]                                              \\
    & \qquad \P[\mathcal C\in \d C] (1-f(x))\d x.
  \end{aligned}
\end{equation}
Notice that when conditioning on $\mathcal C = C$ the two events become independent. The former depends only on the edges between $C$ and $x$, while the latter only depends on $\eta\setminus C$ and the edges between itself and $x$. If $x\not\in \Lambda_L$ we have $\P[x\sim C] = 0$. By using Lemma \ref{cor:stop}
(taking $u$ to be the ghost-vetex $g_{KL}$ with $v=o$), translation invariance and observing that for a path starting at $x$
(with $x \in \Lambda_L$)
to reach $\mathbb R^d \setminus \Lambda_{kL}$ it must first reach
$\mathbb R^d \setminus \Lambda_{(k-1)L}(x)$ we bound
\begin{equation}
  \P[
    x\leftrightarrow g_{kL}  
    \text{ in } \xi[(\eta\setminus C)\cup \{x\}]
    \mid
    \mathcal C = C
  ]
  \leq
  \P[o\leftrightarrow g_{(k-1)L} 
  \text{ in } \xi^o].
\end{equation}
Applying this to the inequality we find that
\begin{equation}
  \begin{aligned}
    \theta_{kL}(\lambda)
    & \le
    \lambda
    \int_{\mathbb R^d} \int
    \P[
      o\leftrightarrow x \text{ in } \xi[C\cup \{x\}]
      \mid
      \mathcal C = C
    ]                                                 \\ &\qquad \times
    \P[
      x\leftrightarrow g_{k L} 
      \text{ in } \xi[(\eta\setminus C)\cup \{x\}]
      \mid
      \mathcal C = C
    ]
    \P[\mathcal C\in \d C]
    (1-f(x))
    \d x                                              \\
    & \leq
    \lambda\int_{\mathbb R^d}
    \P[
      o\leftrightarrow x \text{ in } \xi[\mathcal C\cup \{x\}]
    ]
    (1-f(x))
    \d x \times
    \P[
      o\leftrightarrow
      g_{(k-1)L}
      \text{ in } \xi^{o}
    ]
    \\
    & = \varphi_\lambda(f) \theta_{(k-1)L}(\lambda).
  \end{aligned}
\end{equation}
The final line is true simply by definition of $\mathcal C$ and $\varphi_\lambda$.
By iteration we find that $
\theta_{kL}(\lambda)\leq \varphi_\lambda(f)^{k-1}.
$
Using  the fact that $\theta_t$ is decreasing in $t$ we obtain
for $kL \leq t \leq (k+1)L$
with $k \in \mathbb N,  k \geq 3$
that
$$
\theta_t(\lambda) \leq \varphi_\lambda(f)^{k-1} \leq \varphi_\lambda(f)^{t/2}.
$$
Also there is a constant $\delta >0$ such that $\theta_t(\lambda) < 1- \delta$
for all $ t \in (0, 3L)$, and hence
item \ref{itm:first} holds for all $\lambda<\widetilde{\lambda_c}$.

\subsubsection{Item \ref{itm:second}}
We will prove \hyperref[thm:sharpness]{item II} by first showing for all $t\geq1$ the lower bound
\begin{equation}\label{eq:diff_ineq}
  \frac{\d}{\d\lambda}\theta_t(\lambda)
  \geq
  \frac{1}{\lambda}\big(\inf_{f\in\mathcal T}\varphi_\lambda(f)\big) (1-\theta_t(\lambda)).
\end{equation}

Write $u\leftrightarrow v$ \emph{through} $x$ if \emph{every} possible path from $u$ to $v$ passes through $x$. We apply Lemma \ref{cor:stop}, with $v$ taken
to be the ghost-vertex $g_t$,
to get the `outside component' (see equation \eqref{eq:ghost}). Write $\mathcal D = C(g_t,\xi^{o,g_t})$ for the outside component (without $x$). The first equality in the next display comes from the \hyperref[par:Margulis]{Margulis-Russo formula} from Section \ref{ssec:standard_tools}. We then marginalize over possible configurations of $\mathcal D$. Finally, we see that for $o\leftrightarrow g_t$ through $x$ we require every path $o\leftrightarrow x$ to avoid $\mathcal D$, and for $x$ to connect to $\mathcal D$. In symbols
\begin{equation}
  \begin{aligned}
    \frac{\d}{\d\lambda}\theta_t(\lambda, \psi)
    & =
    \int_{\Lambda_{t+2}}
    \P[o\leftrightarrow g_t \text{ through } x \text{ in }
      \xi[\eta\cup\{o,x,g_t\}]
      \d x                       \\
      & =
      \int_{\Lambda_{t+2}}\int
      \P[
        o\leftrightarrow g_t \text{ through } x \text{ in } \xi^{o,x,g_t}
        \mid
        \mathcal D = C
      ]
      \P[\mathcal D\in \d C]\d x \\
      & =
      \int_{\Lambda_{t+2}}\int
      \P[
        o\leftrightarrow x\text{ in } \xi[(\eta\setminus C)\cup\{o,x\}],
        x\sim C
        \mid
        \mathcal D = C
      ]                          \\&\qquad\qquad
      \mathbf{1}\{o\not\in C\}
      \P[\mathcal D\in \d C]\d x
    \end{aligned}
  \end{equation}
  Next, we observe that events  $\{o\leftrightarrow x\}$ and
  $\{x\sim C\}$ are conditionally independent, as they rely on a disjoint set of edges. Then we will be able to apply Lemma \ref{cor:stop}. We first define
  \begin{equation}
    \mathcal T_t :=
    \{
      f:\reals^d\to[0,1] \text{ measurable}
      \mid
      \text{supp}(f)\subset\Lambda_t
    \},
  \end{equation}
  over which we will take an infimum. By definition $g_t$ is always in $\mathcal D$. For any possible configuration $C$ of $\mathcal D$ we have $\P[x\sim C]\geq \conn(x,g_t)$, and in particular
  $f^C(x) = 0 $ for $x \in \Lambda_{t+2} \setminus \Lambda_t$.
  Hence,

  \begin{align}
    \frac{\d}{\d\lambda}\P[o\leftrightarrow g_t]
    & =
    \int_{\Lambda_{t+2}}\int
    \P[
      o\leftrightarrow x\text{ in } \xi[(\eta\setminus C)\cup\{o,x\}]
      \mid
      \mathcal D = C
    ]                          \\
    & \qquad \qquad \mathbb P[
      x\sim C
    ]
    \mathbf{1}\{o\not\in C\}
    \P[\mathcal D\in \d C]\d x
    \\
    & =
    \int_{\Lambda_{t+2}}\int
    \P[
      o\leftrightarrow x\text{ in }
      \xi[f_*^C\eta\cup\{o,x\}]
    ]
    (1 - f^C(x))
    \mathbf{1}\{o\not\in C\}
    \P[\mathcal D\in \d C]\d x \\
    & \geq
    \inf_{f\in\mathcal T_t}
    \int_{\Lambda_{t+2}}\int
    \P[
      o\leftrightarrow x\text{ in }
      \xi[f_*\eta\cup\{o,x\}]
    ]
    (1 - f(x))
    \mathbf{1}\{o\not\in C\}
    \P[\mathcal D\in \d C]\d x.
  \end{align}
  We can expand the domain of integration to all of $\reals^d$, since the integrand is $0$ outside of $\Lambda_{t+2}$. Since $\mathcal T_t\subset \mathcal T$, we can  bound the $\inf_{f\in\mathcal T_t}$ by $\inf_{f\in\mathcal T}$ from below. We find that
  \begin{align}
    \frac{\d}{\d\lambda}\P[o\leftrightarrow g_t]
    & \geq
    \frac1\lambda
    \inf_{f\in\mathcal T_t}
    \lambda
    \int_{\reals^d}
    \P[
      o\leftrightarrow x\text{ in }
      \xi[ f_*\eta\cup\{o,x\}]
    ]
    (1 - f(x))
    \d x                        \\
    & \qquad\qquad\times
    \int \mathbf{1}\{o\not\in C\}
    \P[\mathcal D\in \d C]      \\
    & \geq \label{eq:integral}
    \frac1\lambda
    \inf_{f\in\mathcal T} \varphi_\lambda(f)
    \int
    \mathbf{1}\{o\not\in C\}
    \P[\mathcal D\in \d C]      \\
    & =
    \frac1\lambda
    \inf_{f\in\mathcal T} \varphi_\lambda(f)
    \P[o\nleftrightarrow g_t],
  \end{align}
  where to get \eqref{eq:integral} from the previous line we apply the definition
  \eqref{e:defphi} of $\varphi_\lambda$ and take the $\inf$ over the larger set $\mathcal T$. This gives us \eqref{eq:diff_ineq}.

  We now derive item \ref{itm:second} from \eqref{eq:diff_ineq}. Note that by the
  definition of $\tilde\lambda_c$ at \eqref{e:tlacdef}
  we know that for all $\lambda>\tilde\lambda_c$ and all $f\in\mathcal T$ we have $\varphi_{\lambda}(f)\geq 1$. Then
  \begin{equation}
    \frac\d{\d\lambda}\theta_t(\lambda)
    \geq
    \frac1\lambda (1-\theta_t(\lambda)).\\
  \end{equation}
  We divide both sides by $1-\theta_t$ and integrate from $\tilde\lambda_c$ to $\lambda$ to obtain
  \begin{equation}
    \int_{\tilde\lambda_c}^\lambda
    \frac{\theta'_t(\lambda')}{1-\theta_t(\lambda')}\d \lambda'
    \geq
    \int_{\tilde\lambda_c}^\lambda
    \frac1{\lambda'}
    \d \lambda'.
  \end{equation}
  By $u$-substitution with $u=1-\theta_t(\lambda')$ we recover
  \begin{equation}
    -\log(1-\theta_t(\lambda)) + \log(1-\theta_t(\tilde\lambda_c))
    \geq
    \log(\frac\lambda{\tilde\lambda_c}).
  \end{equation}
  By rearranging we get $\tilde\lambda_c (1-\theta_t(\tilde\lambda_c))\geq \lambda (1-\theta_t(\lambda))$, and hence
  \begin{equation}
    \theta_t(\lambda)
    \geq
    \frac
    {\lambda-\tilde\lambda_c(1-\theta_t(\tilde\lambda_c))}
    {\lambda}.
  \end{equation}
  Now we can let $t\to\infty$ and using the fact that $(1-\theta(\tilde\lambda_c))\leq 1$ we recover
  \begin{equation}
    \theta(\lambda)
    \geq
    \frac
    {\lambda-\tilde\lambda_c}
    \lambda.
  \end{equation}

  \section{Proof of Theorem \ref{thm:main_large}}\label{sec:LargeComp}
  We now have the tools to show Theorem \ref{thm:main_large}.
  We assume throughout this section that $\text{supp}(\psi) \subset \Lambda_1$
  (again, it is straightforward to adapt the arguments to general  finite-range $\psi$).
  Remember we are trying to show for $\lambda\in(0,\lambda_c)$ that
  $$
  \frac{|L_1(\xi_{\lambda}\cap \Lambda_t)|}{\log t}
  \longrightarrow
  \frac d {\zeta(\lambda)} \ \text{ in probability.}
  $$

  We will first use exponential decay of the $t$-percolation probability to prove exponential decay of the number of vertices in the component containing the origin. We then use this fact to show that the volume decay rate
  $\zeta(\lambda)$ is well defined, continuous and decreasing.

  \subsection{Exponential decay in volume}
  \label{ss:expvol}
  To show the volume decay rate is strictly positive,
  we need an exponential bound on the number of vertices in the component containing the origin. We provide this in Lemma \ref{lem:vol_decay} below.
  Before proving this in detail,  we describe our strategy. We divide
  $\mathbb R^d$ into boxes of large fixed side-length $\kappa$. We say a box
  is {\em good} if there is a path in the RCM from inside the box to
  some other box exactly 2 steps away (allowing diagonal steps).
  By choosing $\kappa$ large enough, using Lemma
  \ref{lem:boxboxcrossing} below (a consequence of sharpness) we can make the probability of a box being good very small. Then by classical path-counting arguments the  cluster at the origin  of good boxes has an exponentially decaying tail.
  By a Chernoff estimate on the Poisson distribution, the probability of a  particular cluster of boxes containing
  a larger-than expected number of Poisson points
  also has an exponentially decaying tail, and combining these two exponential tail estimates gives us the result.


  \begin{lemma} \label{lem:boxboxcrossing}
    Suppose $\lambda<\lambda_c$. Let $b\geq a > 0$. Then
    \begin{equation} \label{e:annulus}
      \P_{\lambda}[\Lambda_a\leftrightarrow (\mathbb
        R^d \setminus \Lambda_b)
      ]
      \leq
      \lambda a^d \exp(-c (b-a))
    \end{equation}
    where $c=c(\lambda)$ is the same constant as in
    Theorem \ref{thm:sharpness}.
  \end{lemma}

  \begin{proof}
    We apply the Markov inequality followed by the Mecke equation. We find that
    \begin{align*}
      \P_{\lambda}[
        \Lambda_a\leftrightarrow (\mathbb R^d \setminus \Lambda_{b})
      ]
      & \leq
      \E_{\lambda} \Big[
        \sum_{x\in\mathcal P_{\lambda} \cap \Lambda_a}
        \mathbf{1}_{\{ x\leftrightarrow (\mathbb R^d \setminus \Lambda_{b})
        \text{ in } \xi_{\lambda} \}}
      \Big]                                                                                    \\
      & = \lambda \int_{\Lambda_a}
      \P[x\leftrightarrow (\mathbb R^d \setminus \Lambda_{b}) \text{ in } \xi_\lambda^x]
      \d x                                                                                 \\
      & \leq \lambda \text{Leb}(\Lambda_a) \P_{\lambda}[o\leftrightarrow
      (\mathbb R^d \setminus \Lambda_{b-a}) \text{ in } \xi_\lambda^x ],
    \end{align*}
    where Leb denotes Lebesgue measure.
    Applying
    Theorem \ref{thm:sharpness} we obtain \eqref{e:annulus}.
  \end{proof}

  \begin{lemma}[Exponential decay of $|C_o|$] \label{lem:vol_decay}
    Given $\lambda\in(0,\lambda_c)$ there exist constants $\tilde C,\tilde c>0$ such that for $n$ sufficiently large we find
    \begin{equation}
      \P[|C_o|\geq n]\leq \tilde C\exp(-\tilde c n).
      \label{e:expCo}
    \end{equation}
  \end{lemma}

  \begin{proof}
    We implement the strategy described above.
    The first part of the argument is similar to
    e.g.  \cite[Theorem 3.7]{duminil-intro} but the last part is specific to our continuum setting. Recall the notation $\Lambda_k(x) := x + [-\frac k2,\frac k2]^d$. We remind the reader that the notation $C_x$ refers to the connected component of $\xi_\lambda^x$ containing $x$.

    Let $\kappa > 4$, to be chosen later. We define a new graph which we will call $\mathcal H_\kappa$ with vertex set $\kappa \ints^d$ and edges between vertices $x,y\in \mathcal H_\kappa$ if and only if $\|x-y\|_\infty \leq\kappa$; this graph has degree $D=3^d-1$, where in particular $D$ is independent of $\kappa$.

    We call a finite connected set of vertices of $\mathcal H_\kappa$ an \emph{animal}. We denote by $\mathcal A(m)$ the set of animals that contain the origin and are of cardinality $m$. For each animal $A\in\mathcal A(m)$ let $T(A)$ be a maximal stable set of sites in $A$. That is the largest collection of $x\in A$ such that no two $x,y\in T(A)$ share an edge. In the case of multiple possible such sets any tie breaker will work.

    We shall say that a vertex $x$ in $\mathcal H_\kappa$ is \emph{good} if the event $\{\Lambda_{\kappa}(x)\leftrightarrow \Lambda_{3 \kappa/2}(x)^c \inxi{\lambda}^o\}$ occurs. We shall say an animal $A$ is good if every $x\in A$ is good.

    Let $c'\in(0,1)$, to be chosen later.
    We set $m:=m(n) := \lfloor \frac{c'n}{\lambda\kappa^d}\rfloor$. Let $F_m$ be the event that there exists some $A\in\mathcal A(m)$ such that $A$ is good. Then
    \begin{align}
      \P[|C_o|\geq n]
      \leq
      \P[ F_m] + \P[|C_o|\geq n ,  F_m^c].
      \label{e:splitevent}
    \end{align}

    First we bound $\P[|C_o|\geq n ,  F_m^c]$.
    We define $C^*_o:=\{x\in\kappa\ints^d\ |\ \Lambda_\kappa(x)\cap C_o\neq \varnothing\}$, to be the minimal animal that contains $C_o$.
    Then $C_o^*$ is connected, $o \in C_o^*$ and if $|C_o^*| > D+1$, then
    all sites in $C_o^*$ are good.
    It follows that if $m \geq D+1$ and $|C_o^*| \geq m$ then $F_m$ occurs.
    Also if $|C_o^*| < m$ then there exists at least one
    animal $A \in {\cal A}(m)$ with $C_o^* \subset A$.
    Hence, if $n$ is large enough so that $m \geq D+1$, then
    %

    %
    %
    %
    \begin{align*}
      \P[|C_o|\geq n, F_{m}^c]
      & \leq
      \P[|C_o| \geq n, |C_o^*| < m]
      \\
      & \leq
      \P\bigg[
        \bigcup_{A\in\mathcal A(m)}
        \{ \eta\big(\cup_{x\in A} \Lambda_\kappa(x)\big)\geq n
        \}
      \bigg]  \\
      & \leq
      \sum_{A\in\mathcal A(m)}
      \P[\text{Pois}(\lambda \kappa^d|A|)\geq n].
    \end{align*}
    By \cite[Lemma 9.3]{penrose2003random} we know that $|\mathcal A(m)|\leq 2^{Dm}$. Hence,
    \begin{align*}
      \P[|C_o|\geq n, F_{m(n)}^c]
      & \leq
      2^{mD}
      \P[\text{Pois}(\lambda m \kappa^d)\geq n] \\
      & \leq
      2^{mD}
      \P[\text{Pois}(c'n)\geq n].
    \end{align*}
    Assume $c' \leq e^{-4}$. Then
    by
    \cite[Lemma 1.2, eq. (1.12)]{penrose2003random}
    we have that $ \P[\text{Pois}(c'n)\geq n] \leq e^{-2n}$.
    Assume also that $c' \leq \lambda \kappa^d/(D \log 2)$; then
    $$
    2^{mD} \leq \exp( (D \log 2) c'n /(\lambda \kappa^d) ) \leq e^{n}.
    $$
    Combining these estimates yields
    \begin{equation}\label{eq:exp_bound_eq1}
      \P[|C_o|\geq n, F_m^c]
      \leq
      \exp(-n).
    \end{equation}

    Next we bound $\P[F_m]$. We use the fact that the goodness of sites in $T(A)$ are independent, since the boxes that define them do not intersect.
    It is always possible to find a set $T(A)$ of cardinality at least $(m-1)/D$. Such a set can be constructed iteratively by using a `greedy' approach; subsequently adding vertices adjacent to neighbors of already added vertices in such a way there they are not a direct neighbor of any already chosen vertex. Then by the union bound,
    \begin{align*}
      \P[F_m]
      & \leq
      \sum_{A\in\mathcal A(m)}
      \P[A\text{ is good}  ]                                               \\
      & \leq
      \sum_{A\in\mathcal A(m)}
      \P[\forall x\in T(A): x\text{ is good}]                             \\
      & =
      \sum_{A\in \mathcal A(m)}
      \P[\Lambda_{\kappa} \leftrightarrow (\mathbb R^d \setminus
      \Lambda_{3\kappa/2})]^{|T(A)|} \\
      & \leq
      |\mathcal A(m)| \P[\Lambda_\kappa \leftrightarrow
      (\mathbb R^d \setminus \Lambda_{3\kappa/2})]^{m/D}.
    \end{align*}
    We can now apply Lemma \ref{lem:boxboxcrossing} and pick $\kappa>2$ sufficiently large such that
    $$
    \P[\Lambda_{\kappa}\leftrightarrow (\mathbb R^d \setminus
    \Lambda_{3\kappa/2})]
    \leq
    \lambda \kappa^d \exp(-c\kappa/2)
    \leq
    \frac1e\cdot 2^{-D^2}.
    $$
    Using $|A(m)| \leq 2^{Dm}$ again, we thus
    find that
    \begin{equation}\label{eq:exp_bound_eq2}
      \begin{aligned}
        \P[F_m]
        & \leq
        2^{mD}
        (
          e^{-m/D}
          2^{-mD}
        )
        \leq
        \exp(-c'n/(2 \lambda \kappa^d D)).
      \end{aligned}
    \end{equation}
    Putting \eqref{e:splitevent}, \eqref{eq:exp_bound_eq1} and \eqref{eq:exp_bound_eq2} together gives us \eqref{e:expCo}.
  \end{proof}

  \subsection{Existence and properties of the volume decay rate}\label{sec:LogDecay}
  To prove
  \autoref{thm:main_large}, the only missing ingredient is a more detailed understanding of the volume decay rate $\zeta$. To get this,
  we shall use the following lemmata
  from
  \cite{feketeWurz} and
  \cite{dwass} respectively.
  \begin{lemma}[Fekete's Subadditivity Lemma] \label{lem:Fekete}
    Let $(u_n)_n\subset \reals$ be a sequence of numbers. If for all $n$ and $m$ we have $u_{n+m}\leq u_n + u_m$, then
    $$
    \lim_{n\to\infty} \frac{u_n}{n}
    =
    \inf_n \frac{u_n}{n}
    \in
    [-\infty, \infty).
    $$
  \end{lemma}
  \begin{lemma}[Dwass' formula]\label{thm:dwass}
    Let $\tau$ be a Galton-Watson tree with offspring distribution $\nu$. Let $|\tau|$ be its total progeny. We let $N_k := X_1+\dots+X_k$ where $X_j\sim \nu$ iid.  Then
    \begin{equation}
      \P[|\tau| = k]=\frac1k \P[N_k=k-1].
    \end{equation}
  \end{lemma}

  We recall Definition \ref{def:cor-length}.
  \begin{align}
    \zeta(\lambda):=
    \zeta(\lambda, \conn)
    : & =
    \lim_{n\to\infty}
    -\frac1n \log \P[|C_o|= n\text{ in }\xi_{\lambda}^o]\tag*{\eqref{eq:def-cor-1}}
    \\
    & =
    \lim_{n\to\infty}
    -\frac1n \log \P[
      n\leq|C_o|< \infty \text{ in }\xi_{\lambda}^o
    ].\tag*{\eqref{eq:def-cor-2}}
  \end{align}

  \begin{theorem} \label{thm:zeta-prop}
    Let $\lambda\in(0,\lambda_c)$.
    The volume decay rate $\zeta(\lambda)$ is well defined in
    \eqref{eq:def-cor-1} and \eqref{eq:def-cor-2}.
    It is continuous and non-increasing in $\lambda$. As $\lambda$ goes to zero $\zeta(\lambda)\to \infty$.
  \end{theorem}

  In \cite[Theorem 10.1]{penrose2003random} the above is shown for the special case where $\conn(x)=\mathbf{1}\{\|x\|\leq 1\}$. Our proof in the general case follows a similar scheme to that in \cite{penrose2003random},
  with some differences that we shall point out as we go along.
  (Note that the e-book version of \cite{penrose2003random} contains many typos that are not present in the print version.)
  To prove Theorem \ref{thm:zeta-prop},
  first we show convergence of \eqref{eq:def-cor-1}. We then show equality of \eqref{eq:def-cor-1} and \eqref{eq:def-cor-2}.  Then we show continuity and the limiting behavior as $\lambda$ goes to zero.
  \begin{lemma}[$\zeta$ is well defined]\label{lem:conv-zeta}
    The limit \eqref{eq:def-cor-1} exists, and $\zeta(\lambda)\in[0,\infty)$.
  \end{lemma}
  \begin{proof}
    To prove convergence we use \nameref{lem:Fekete}. We define $p_n:=\P[|C_o(\xi_{\lambda}^o)|= n]$. Showing subadditivity for $-\log p_n$ is equivalent to showing supermultiplicativity for $p_n$. We will show $p_{n+1}p_{m+1}\leq p_{n+m+1}$ which implies Lemma \ref{lem:conv-zeta} by a simple reindexing via $q_n=p_{n+1}$.

    We know by \cite{penrose1991} (see also \cite[Proposition 6.2]{meester1996continuum}) that
    $$
    p_{n+1}
    =
    {\lambda^n}
    \int_{(\mathbb R^d)^n}
    h_2(o,x_1,\dots, x_n)
    \times
    \exp \Big(
      -\lambda
      \int_{\mathbb R^d}
      h_1(y; o, x_1,\dots, x_n)
      \d y
    \Big)
    \d\vec x,
    $$
    where $h_1(y;z_1,\dots, z_k)$ is the probability that $y$ is not isolated from the points $z_1,\dots, z_k$ in $\xi[\{y, z_1,\dots, z_n\}]$ and $h_2(z_1,\dots, z_k)$ is the probability that $z_1,\dots, z_k$ are in one connected component and ordered $z_2<\dots<z_n$ for some order (e.g. lexicographic) in $\xi[\{z_1,\dots, z_k\}]$.
    Note that in the situation of \cite[Theorem 10.1]{penrose2003random}
    the functions $h_1$ and $h_2$ are indicator functions, whereas here they have to be considered as probabilities.

    We use $\vec z$ to represent $(z_1,\dots, z_k)$.
    Combining we find
    \begin{align*}
      p_{n+1}p_{m+1}
      & =
      \lambda^{n+m}
      \int_{(\mathbb R^d)^n}
      \int_{(\mathbb R^d)^m}
      h_2(o,x_1,\dots, x_n)
      h_2(o,y_1,\dots, y_m)
      \\&\quad
      \times    
      \exp\bigg(
        -\lambda
        \Big[
          \int_{\mathbb R^d}
          h_1(z;o,x_1,\dots,x_n)
          \d z
          +
          \int_{\mathbb R^d}
          h_1(z;o,y_1,\dots,y_m)
          \d z
      \Big]\bigg)
      \d \vec y \d\vec x \\
      & =               
      \lambda^{n+m}
      \int_{(\mathbb R^d)^n}
      \int_{(\mathbb R^d)^m}
      h_2(o,x_1,\dots, x_n)
      h_2(x_n,y_1,\dots, y_m)
      \\&\quad
      \times    
      \exp\bigg(
        -\lambda
        \Big[
          \int_{\mathbb R^d}
          \big(
            h_1(z;o,x_1,\dots,x_n)
            +
            h_1(z;x_n,y_1,\dots,y_m)
          \big)
          \d z
      \Big]\bigg)
      \d \vec y \d\vec x
    \end{align*}
    The second equality follows from translation invariance and substitution $y_i\to y_i+x_n$. Connecting $m$ different points to $o$ has the same probability as connecting $m$ different points to $x_n$ when they are all shifted by $x_n$. We rename $y_1,\dots, y_m$ to $x_{n+1},\dots, x_{n+m}$. We now reduced the problem to showing
    $$
    h_2(o,x_1,\dots, x_n)
    h_2(x_n ,x_{n+1},\dots, x_{n+m})
    \leq
    h_2(o, x_1,\dots, x_{n+m})
    $$
    and
    $$
    h_1(z; o, x_1,\dots,x_n)
    +
    h_1(z; x_n,\dots, x_{n+m})
    \geq
    h_1(z; o, x_1,\dots, x_{n+m})
    $$
    The first equation follows, since the union of two connected graphs with a vertex in common is again a connected graph and the ordering is maintained. The second equation holds by the union bound.

    We now get existence by choosing $u_n := -\log\P_{\lambda}[|C_o| = n]$ and applying Lemma \ref{lem:Fekete}. Furthermore, since the $u_n$'s are lower bounded by $0$ we get the stronger bound that $\zeta(\lambda)\in [0,\infty)$.
    We also find immediately by Lemma \ref{lem:vol_decay} that if $\lambda <\lambda_c$, then for $n$ sufficiently large we have
    \begin{equation}\label{eq:usefek}
      -\frac1n \log \P [ |C_o| =n ]
      \geq -\frac1n \log( \tilde C\exp(-\tilde cn))
      \geq
      \tilde c-\frac{\log\tilde C}{n}
      \xrightarrow{n\to\infty} \tilde c>0,
    \end{equation}
    and hence $\zeta(\lambda)>0$:
  \end{proof}

  \paragraph{Equivalence of definitions}
  Next, we show that the two different definitions \eqref{eq:def-cor-1} and \eqref{eq:def-cor-2}, of the volume decay rate are indeed equivalent; here we give some details that were omitted in \cite{penrose2003random}.
  It suffices to show that
  $$
  q_n:=q_n(\lambda):=
  \Big(
    \frac{
      \P[|C_o|=n]
    }{
      \P[n\leq |C_o|< \infty]
    }
  \Big)^\frac{1}{n}
  \xrightarrow{n\to\infty} 1.
  $$
  We immediately know that $q_n\leq 1$, so we only need to bound it from below. Let $0<\zeta^-<\zeta(\lambda)<\zeta^+$, to be chosen later. For now $\zeta(\lambda)$ will refer to the definition using $\P[|C_o|=n]$, i.e. \eqref{eq:exp_bound_eq1}. Then, by existence of the limit (Lemma \ref{lem:conv-zeta}), we know that for $n$ large enough
  \begin{equation} \label{eq:eq-one}
    e^{-n\zeta^+}
    <
    \P[|C_o|=n]
    <
    e^{-n\zeta^-}.
  \end{equation}
  This gives us the following bounds on
  $\P[n\leq|C_o|< \infty]=\sum_{k\geq n}\P[|C_o|=k]$,
  \begin{equation}\label{eq:ineq-one}
    \frac{e^{-\zeta^+n}}{1-e^{-\zeta^+}}
    <
    \P[n\leq |C_o|< \infty]
    <
    \frac{e^{-\zeta^-n}}{1-e^{-\zeta^-}}.
  \end{equation}
  For the above bound we used the fact that $\zeta^->0$ to obtain convergence for the geometric series. This ceases to be possible if $\zeta(\lambda)=0$. This, however, only happens in the critical and supercritical cases, which we disregard here.

  Now let $\eps>0$. We want to show the existence of an $n_0$ such that for all $n\geq n_0$ we have $q_n\in(1-\eps,1]$. We now choose $\zeta^\pm =\zeta(\lambda)\pm \eps/4$. Let $n$ be sufficiently large, such that \eqref{eq:eq-one} holds and
  \begin{equation}\label{eq:zeta_eps_simp}
    e^{-\eps/2}(1-e^{-\zeta^-})^\frac{1}{n}
    >
    (1-\eps).
  \end{equation}
  Now we bound $q_n$ using \eqref{eq:eq-one} and \eqref{eq:ineq-one}:
  \begin{equation*}
    q_n
    >
    \frac
    {e^{-\zeta^+}(1-e^{-\zeta^-})^{\frac1n}}
    {e^{-\zeta^-}}.
  \end{equation*}
  Substituting for the defintion of $\zeta^\pm$ together with \eqref{eq:zeta_eps_simp} we find that
  \begin{align*}
    q_n
    >
    1-\eps.
  \end{align*}
  Thus the two definitions of $\zeta$ are equivalent.

  \paragraph{Continuity and monotonicity}
  For continuity and monotonicity of $\zeta$ we follow \cite[Theorem 10.1]{penrose2003random} quite closely.
  First we show that $\zeta$ is non-increasing and continuous. Consider the quantities $q_n(\lambda) = \P[|C_o|=n]$ and $q^+_n:= \P[n\leq| C_o|< \infty]$. It is easy to see that $q^+_n$ is increasing in $\lambda$ for every $n$ in the subcritical regime. We now define
  $$\rho(\lambda)
  :=
  \lim_{n\to\infty} q^+_n(\lambda)^{1/n}
  =
  e^{-\zeta(\lambda)}.
  $$
  We see that $\rho$ is non-decreasing, which in turn shows that $\zeta$ is non-increasing in $\lambda$.

  We move on to continuity in $\lambda$. We couple the RCM at different intensities $\lambda$ by giving every point $X_i$ a marking $W_i$, as described in Section \ref{ss:DefOver}.
  For a given intensity $\lambda$ we simply retain all points where $W_i\leq \lambda$.

  Consider $0<\lambda<\mu<\lambda_c$.
  Our coupling ensures that all edges of $\xi^o_\lambda$ will also
  be present  in $\xi^o_\mu$ (in the situation of \cite{penrose2003random}
  this would have been automatic).

  One way for $|C_o|=n$ in $\xi^o_{\lambda}$ to hold is to require $|C_o|=n$ in $\xi^o_{\mu}$ and all $(X_i, W_i)\in C_o$ have the property that $W_i<\lambda$, which has probability $\lambda/\mu$, per vertex.
  This gives
  \begin{equation}
    \P[|C_o|=n \text{ in } \xi^o_{\lambda}]
    \geq
    (\frac\lambda\mu)^{n-1}
    \P[|C_o|=n \text{ in } \xi^o_{\mu}].
  \end{equation}
  Then
  $$
  \rho(\lambda)
  =
  \lim_{n\to\infty} q_n(\lambda)^{1/n}
  \geq
  \lim_{n\to\infty} ((\frac\lambda\mu)^{n-1} q_n(\mu))^{1/n}
  =
  \frac\lambda\mu \rho(\mu).
  $$
  This together with the fact that $\rho$ is non-decreasing gives the continuity for $\rho$. Fix some $\lambda\in(0,\lambda_c)$. Let $\eps>0$, $\delta = \lambda\eps/{\rho(\lambda)}$ and $\lambda'\in (\lambda-\delta, \lambda + \delta)$. Assume $\lambda' > \lambda$. Then
  $$
  \rho(\lambda')-\rho(\lambda)
  \leq
  \frac{\lambda'}\lambda \rho (\lambda) - \rho(\lambda)
  \leq \rho(\lambda)\frac{\delta}{\lambda}
  \leq \eps.
  $$
  The case where $\lambda'<\lambda$ is analogous.

  Continuity of $\rho$ implies that $\zeta$ is also continuous in $\lambda$ in the range $(0,\lambda_c)$.

  \paragraph{Near-zero behavior}
  We show as $\lambda\to 0$ that $\zeta(\lambda, \conn)\to\infty$.
  The proof of this in \cite{penrose2003random} for the Gilbert case does not
  appear to carry over, and we give a different argument.
  We show this by bounding the size of the component of the origin by the total progeny (total number of vertices) of a Galton-Watson tree.
  %
  For simplicity we write
  $\int \psi dx$ for $\int_{\mathbb R^d} \psi(x)dx$, and set
  $\alpha :=\lambda\int\conn\d x$.
  We know
  that the origin has
  $\Pois(\alpha)$
  many offspring.
  Let $\tau$ denote the Galton-Watson tree with offspring distribution $\Pois(\alpha)$. We know
  from e.g.  \cite{spread-out}
  that $|\tau|$ dominates $|C_o|$ in the sense that
  \begin{equation}
    \P[|\tau|\geq k]\geq \P[|C_o|\geq k] ~~~~ \forall ~ k \in \mathbb R.
  \end{equation}
  We know that for $\alpha<1$ the Galton-Watson tree is subcritical and so the total progeny $|\tau|$ of the tree is almost-surely finite.
  By using \nameref{thm:dwass} (Lemma \ref{thm:dwass}) we find that
  \begin{equation}
    \begin{aligned}
      \P[|\tau|\geq k]
      & =
      \sum_{j=k}^\infty
      \frac1j\P[\Pois(j\alpha)=j-1] \\
      & =
      \sum_{j=k}^\infty
      \frac{
        e^{-j\alpha}(j\alpha)^{j-1}
      }{
        j!
      }.
    \end{aligned}
  \end{equation}
  By the fact that for all $j\in\nats$ we have $\frac{j^j}{j!}\leq e^j$ we find that
  \begin{equation}
    \begin{aligned}
      \P[|\tau|\geq k]
      & \leq
      \frac1\alpha
      \sum_{j=k}^\infty
      (e^{(1-\alpha)}\alpha)^j \\
      & =
      \frac{
        (e^{1-\alpha}\alpha)^k
      }{
        \alpha(1-e^{1-\alpha}\alpha)
      }.
    \end{aligned}
  \end{equation}
  By substituting $\lambda\int\conn\d x$ back in for $\alpha$ we find that
  \begin{equation}
    \begin{aligned}
      \limsup_{n\to\infty}
      \frac1n \log \P[|C_o|\geq n]
      & \leq
      \limsup_{n\to\infty}
      \frac1n \log \P[|\tau|\geq n] \\
      & \leq
      \limsup_{n\to\infty}
      \Big(
        (1-\alpha)
        + \log\alpha
        -\frac1n \log\alpha
        -\frac1n \log(1-e^{1-\alpha}\alpha)
      \Big)                         \\
      & \leq
      1-\lambda\int\psi\d x + \log\big(\lambda \int \psi \d x\big).
    \end{aligned}
  \end{equation}
  And so the upper bound tends to $-\infty$ as $\lambda$ goes to zero. Hence, by applying Definition \ref{def:cor-length}, we have that $\zeta(\lambda, \conn)\to\infty$, completing the proof of Theorem \ref{thm:zeta-prop}.

  \subsection{Proof of \autoref{thm:main_large}}
  The following proof closely follows the proof in
  \cite[Theorem 10.3]{penrose2003random} for the special case of
  the Gilbert graph.
  \begin{proof}[Proof of \autoref{thm:main_large}]
    We start by showing that the largest component in a box with side lengths $s$ is no larger than $\frac{d}{\zeta(\lambda)}\log s$. Let $\alpha > \frac{d}{\zeta(\lambda)}$. By applying the Markov bound and then the Mecke formula we find
    \begin{align*}
      \P[
        |L_1(\mathcal \xi_{\lambda}\cap \Lambda_s)|
        \geq
        \alpha \log s
      ]
      & \leq
      \E \Big[
        \sum_{x\in\mathcal P_\lambda}
        \mathbf{1}\{|C(x, \xi^x\cap\Lambda_s)|\geq \alpha\log s\}
      \Big]
      \\
      & =
      \lambda\int_{\Lambda_s}
      \P[
        |C(x, \xi^x\cap\Lambda_s)|\geq\alpha \log s \text{ in }\xi^x_{\lambda}
      ]
      dx
    \end{align*}
    Now consider some $\zeta'\in(\frac d\alpha,\zeta(\lambda))$. By the
    definition \eqref{eq:def-cor-2} of $\zeta(\lambda)$ as a limit, we know that for $n$ and $s$ large enough
    \begin{align*}
      \P[|C(x, \xi^x\cap\Lambda_s)|\geq\alpha \log s\text{ in }\xi^x_{\lambda}]
      & \leq
      \P[|C_o|\geq \alpha \log s\text{ in }\xi^o_{\lambda}]
      \\
      & \leq
      \exp(-\zeta'\alpha \log s) = s^{-\zeta'\alpha}.
    \end{align*}
    Hence, by our choices of $\alpha$ and $\zeta'$, we find
    \begin{align*}
      \P[|L_1(\mathcal \xi_{\lambda}\cap \Lambda_s)|
        \geq
        \alpha \log s
      ]
      \leq
      \lambda s^{d-\zeta'\alpha}\xrightarrow{s\to\infty}0.
    \end{align*}

    For the other direction we choose $\beta < {d}/{\zeta(\lambda)}$ and $\zeta''\in(\zeta(\lambda), d / \beta)$.
    We will tile the box $\Lambda_s$ with smaller boxes.
    Let $m(s) = \lfloor\frac{s}{2\beta\log s}\rfloor^d$ denote the number of boxes.
    Let $\{B_{1,s},\dots, B_{m(s), s}\}$ be the maximal collection of disjoint boxes with side-length $s/\sqrt[d]{m(s)}\geq 2\beta \log s$.
    Thus, the boxes $B_{i,s}$ fill $\Lambda_s$ exactly.
    Let $x_{i,s}$ denote the center of the box $B_{i,s}$.
    Now take $\lambda'\in(0,\lambda)$ such that $\zeta(\lambda')<d/\beta$.
    This is possible by the continuity of $\zeta$.
    We can separate $\mathcal P_\lambda$ into a union of $\mathcal P_{\lambda'}$ and
    $$
    \mathcal P_{\lambda', \lambda} := \mathcal P_\lambda \setminus \mathcal
    P_{\lambda'}
    =
    \set
    {X_i:
    \lambda'\leq W_i<\lambda}.
    $$
    Now take $\zeta'' \in (\zeta(\lambda'),d/\beta)$.

    Let $x\in\reals^d$ and $r>0$. Denote by $D(x, r)$ the closed ball with radius $r$ centered at $x$. If $\mathcal P_{\lambda', \lambda}\cap D(x_{i,s}, 1)$ consists of a single point we denote that point by $X_{i,s}$ and let
    $$
    V_{i,s}
    :=
    |C
    (X_{i,s}, \mathcal \xi^{X_{i,s}}_{\lambda'}\cap B_{i,s})
    |,
    $$
    where $X_{i,s}$ inherits the connections from the original sampling of $\xi_{\lambda}$.
    If $|\mathcal P_{\lambda', \lambda}\cap B_{i,s}|\neq 1$ then let $V_{i,s} = 0$.
    Let $\mu$ be the volume of a $d$-dimensional unit ball. By our choice of box size we know that $\{0<V_{i,s}<\beta\log s\}\subset\{C_{X_{i,s}}(\xi_{\lambda'})\subset B_{i,s}\}$ and so $V_{i,s}$ is has the distribution of the size of the component of the origin. Then, by independence of $\mathcal P_{\lambda'}$ and $\mathcal P_{\lambda', \lambda}$, we find for large $s$ that
    \begin{align*}
      \P[V_{i,s}\geq \beta \log s]
      & \geq
      \mu (\lambda-\lambda')e^{-\mu(\lambda-\lambda')}
      \P[
        |C_o| \geq \beta\log s
        \text{ in }\xi_{\lambda'}^o
      ]                                                           \\
      & \geq c'\exp(-\zeta''\beta \log s) = c's^{-\zeta''\beta},
    \end{align*}
    where the inequality follows from the definition of $\zeta$ and $c'=\mu(\lambda-\lambda')e^{-\mu(\lambda-\lambda')}$. The random variables $V_{i,s}$ are independent, since they are dependent on configurations in disjoint boxes. It follows that
    \begin{align*}
      \P\Big[
        \bigcap_{i=1}^{m(s)} \{V_{i,s}<\beta\log s\}
      \Big]
      & \leq
      \big(
        1-c's^{-\zeta''\beta}
      \big)^{m(s)}
      \\
      & \leq \exp(-c's^{-\zeta''\beta}m(s))
    \end{align*}
    which tends to zero by the definition of $m(s)$ and the fact that $\zeta''\beta < d$. On the other hand, if for some $i$ we have $V_{i,s}\geq \beta \log s$, then $L_1(\xi_{\lambda}\cap \Lambda_s)\geq \beta \log s$. This gives us the desired result.
  \end{proof}

  \section{The supercritical phase}
  \label{s:supercritical}
  To complete the story regarding the largest component away from criticality,
  we provide a law of large numbers in the supercritical phase.
  Given random variables $X$ and $(X_t)_{ t >0}$
  and given $p \in \{1,2\}$,
  we say $X_t \to X$ in
  $\mathcal L^p$ if $\mathbb E[|X_t -X|^p] \to 0$ as $ t \to \infty$.

  \begin{theorem}
    \label{th:super}
    Suppose that  $\psi(\cdot)$ is
    radially symmetric and decreasing with bounded
    support. Suppose $\lambda > \lambda_c$.
    Then as $t \to \infty$,
    \begin{align}
      \label{e:giant}
      t^{-d} |L_1(\xi_\lambda \cap \Lambda_t )|
      \to
      \lambda \theta(\lambda)
      ~ \mathrm{ in } ~ \mathcal L^1.
    \end{align}
  \end{theorem}
  \begin{proof}
    If $d=1$ then $\lambda_c = \infty$ so there is nothing to prove.

    Now suppose $d \geq 2$ and $\lambda > \lambda_c$.
    Set $\gamma:= \gamma(d,\lambda,\psi(\cdot)) := \lambda \theta(\lambda)$.
    Given $t >0$, define the random variable $X_t :=
    |L_1(\xi_\lambda \cap \Lambda_t)|$.

    Suppose $d=2$. It is proved in \cite{pen_giant_comp_srgg}
    that for this case, $t^{-d}X_t \to \gamma $ in probability.
    Let $\varepsilon > 0$.
    By the Cauchy-Schwarz inequality we have
    $$
    \mathbb  E[ |t^{-d} X_t - \gamma |
    \mathbf{1} \{ |t^{-d} X_t - \gamma | > \varepsilon \}]
    \leq  \sqrt{\mathbb E[ (t^{-d} X_t - \gamma)^2]}
    \sqrt{\mathbb P[|X_t -\gamma| > \varepsilon] }
    $$
    which tends to zero.
    Moreover $\mathbb E[X_t^2]
    \leq \mathbb \E[ | \mathcal P_\lambda \cap \Lambda_t|^2
    ] = O(t^{2d})$ as $t \to \infty$, so
    for large enough $t$ we have
    $$
    \mathbb E[ |t^{-d} X_t - \gamma|] \leq \varepsilon
    +
    \mathbb E[ |t^{-d} X_t - \gamma| \mathbf{1}
    \{ |t^{-d} X_t- \gamma| \geq \varepsilon \}] \leq 2 \varepsilon.
    $$
    This gives us \eqref{e:giant} for $d=2$.

    Now suppose $d \geq 3$.
    In \cite{kupper2025} a related  result is proved, namely
    \begin{equation}
      \lim_{t \to \infty} \big(
        t^{-d} \mathbb E[ X_t 
        ]
      \big) = \gamma, 
      \label{e:giant2}
    \end{equation}
    subject to an assumption which is proved in
    \cite[Corollary 1.4]{penrose2025supercritical}.
    We shall 
    deduce \eqref{e:giant} from
    \eqref{e:giant2} (the proof of \eqref{e:giant2} in \cite{kupper2025}
    is substantial, and we make no attempt to reproduce it here).

    Let $Y_t$ be the number of vertices $x \in \mathcal P_\lambda
    \cap \Lambda_t$
    such that $x \leftrightarrow (\mathbb R^d \setminus \Lambda_{t^{1/(2d)}}
    (x) )$ in
    $\xi[\mathcal P_\lambda ]$.
    This random variable is a more manageable proxy
    for $X_t$. By the Mecke formula and stationarity,
    \begin{align}
      \mathbb E[ Y_t] =  \int_{\Lambda_t}
      \lambda \mathbb P[
        x \leftrightarrow (\mathbb R^d \setminus
      \Lambda_{t^{1/(2d)}}(x)) ~ \mathrm{in} ~ \xi_\lambda^x]
      \d x  = t^d  \lambda \theta_t(x).
      \label{e:bymecke1}
    \end{align}
    By the Mecke formula again,
    \begin{align}
      \mathbb E[Y_t(Y_t -1)] =
      \lambda^2 \int_{\Lambda_t} \int_{\Lambda_t}
      \mathbb P[
        \{x \leftrightarrow \mathbb R^d \setminus
        \Lambda_{t^{1/(2d)}}(x) ~ \mathrm{in} ~ \xi_\lambda^{x,y}]
      \}
      &
      \\
      \cap \{ y \leftrightarrow \mathbb R^d \setminus
        \Lambda_{t^{1/(2d)}}(y) ~ \mathrm{in} ~ \xi_\lambda^{x,y}
  ] \}] &
  \d y \d x
  \label{e:bymecke2}
\end{align}
Since
the integrand in \eqref{e:bymecke2}
is equal to $\theta_t(\lambda)^2$ whenever $\| y-x \| \geq
2 \sqrt{d} t^{1/(2d)}$ and lies in $[0,1]$ for all $(x,y)$ we obtain that
$
t^{-2d}  \mathbb E[Y_t(Y_t -1)] \to  (\lambda \theta_t(\lambda))^2
= \gamma^2 $
as $t \to \infty$.
Then using also \eqref{e:bymecke1}   we obtain that
$
\mathbb E[(t^{-d} Y_t - \gamma 
)^2] \to 0.
$
Thus $t^{-d} Y_t \to \gamma $
in ${\mathcal L}^2$ and hence in probability as $t \to \infty$.

Now let $\varepsilon > 0$. If
$t^{-d} X_t
\geq
\gamma +   \varepsilon $ but $t^{-d} Y_t <
\gamma + \varepsilon$, then there exists a component
of $\xi[\mathcal P_{\lambda} \cap \Lambda_t]$ with
at least $t^d(\gamma
+ \varepsilon )$
vertices and with all of its vertices within $\ell^\infty$ distance
$2 t^{1/(2d)}$ of each other (since otherwise they would all contribute to
$Y_t$ and hence $Y_t$ would be too big).
Thus, in this case there is at least one $x \in \mathcal P_\lambda$
such that $|\mathcal P_\lambda \cap \Lambda_{2t^{1/(2d)}}(x)  | \geq
t^d(\gamma + \varepsilon)$.
Hence by Markov's inequality followed by the Mecke equation,
\begin{align}
  &  \mathbb P[
    t^{-d} X_t
    \geq
    \gamma +   \varepsilon , t^{-d} Y_t <
    \gamma 
  + \varepsilon]
  \\
  & \leq
  \mathbb E \Big[ \sum_{x \in \mathcal P_\lambda
    \cap \Lambda_t } \mathbf{1}\{
      | \mathcal P_\lambda \cap \Lambda_{2 t^{1/(2d)}}(x)
      | \geq (\gamma 
      + \varepsilon)
  t^d  \} \Big]
  \\
  & = \lambda \int_{\Lambda_t} \mathbb P[
    | \mathcal P_\lambda^x \cap  \Lambda_{2 t^{1/(2d)}}(x)
    | \geq (\gamma 
    + \varepsilon)
  t^d  ] \d x
\end{align}
and using \cite[Lemma 1.2]{penrose2003random}, we obtain that this
probability tends to zero. Since also $\mathbb P[t^{-d} Y_t
  \geq \gamma 
+ \varepsilon] \to 0 $,
we  have that $\mathbb P[t^{-d} X_t 
  > \gamma 
+ \varepsilon] $ tends to zero. Therefore
by the Cauchy-Schwarz inequality,
\begin{align}
  \mathbb E[ (t^{-d} X_t - \gamma 
    )^+
    \mathbf{1} \{ t^{-d} X_t - \gamma 
      > \varepsilon
  \}]
  \leq \sqrt{ \mathbb E[(t^{-d} X_t - \gamma 
  )^2]}
  \sqrt{\mathbb P[ t^{-d} X_t - \gamma 
  > \varepsilon] }
\end{align}
which tends to zero since $\mathbb E[X_t^2] \leq \mathbb E[|\mathcal P_\lambda
\cap \Lambda_t|^2] = O(t^{2d})$ as $t \to \infty$. Also clearly
$\mathbb E[(t^{-d} X_t - \gamma 
  )^+ \mathbf{1} \{t^{-d} X_t
    - \gamma 
\leq \varepsilon \}] \leq \varepsilon$,
and since $\varepsilon$ is arbitrarily small we have
$\mathbb E[(t^{-d} X_t - \gamma 
)^+ ] \to 0$
as $t \to \infty$. Then using \eqref{e:giant2} we also have as
$t \to \infty$ that
$$
\mathbb E[(t^{-d} X_t - \gamma 
)^- ]
= \mathbb E[(t^{-d} X_t - \gamma 
)^+ ]
- \mathbb E[(t^{-d} X_t - \gamma 
) ]
\to 0,
$$
and hence
$$
\mathbb E[|t^{-d} X_t - \gamma 
| ]
=
\mathbb E[(t^{-d} X_t - \gamma 
)^+ ]
+
\mathbb E[(t^{-d} X_t - \gamma 
)^- ]
\to 0.
$$
Therefore
$t^{-d} X_t$ converges to  $\gamma 
$ in ${\mathcal L}^1$, i.e. 
\eqref{e:giant}
holds.
\end{proof}


\end{document}